\documentclass[11pt,a4paper]{article}

\usepackage[latin1]{inputenc}
\usepackage{amsmath}
\usepackage{amsfonts}
\usepackage{amssymb}
\usepackage{graphicx}
\usepackage{enumerate}
\usepackage{geometry}
\usepackage{hyperref}
\usepackage{cite}
\hypersetup{colorlinks=true,urlcolor=blue,citecolor=blue}
\usepackage{tikz}
\usetikzlibrary{calc}
\usepackage{float}

\usepackage{amsthm}
\usepackage{thmtools}
\usepackage[capitalise]{cleveref}
\newtheorem{theorem}{Theorem}[section]
\newtheorem{proposition}[theorem]{Proposition}
\newtheorem{corollary}[theorem]{Corollary}
\newtheorem{definition}[theorem]{Definition}

\declaretheorem[style=remark,qed=$\Diamond$,Refname={Remark,Remarks},sibling=theorem]{remark}

\declaretheorem[style=remark,qed=$\Diamond$,Refname={Example,Examples},sibling=theorem]{example}

\Crefname{assumption}{Assumption}{Assumptions}

\renewcommand{\epsilon}{\varepsilon}
\newcommand{\paren}[1]{\left(#1\right)}

\newcommand{\ip}[2]{\left\langle #1,~ #2\right\rangle}

\newcommand{\gph}{\operatorname{gph}}

\newcommand{\dom}{\operatorname{dom}}

\newcommand{\epi}{\operatorname{epi}}

\newcommand{\ncone}[1]{{N}_{#1}}

\newcommand{\pncone}[1]{N^{\text{\rm P}}_{#1}} 
\newcommand{\cncone}[1]{N^{\text{\rm C}}_{#1}} 
\newcommand{\lsd}{{\partial\!}}
\newcommand{\csd}{\partial^{\rm C}\!}

\newcommand{\vbar}{{\overline{v}}}

\newcommand{\xbar}{{\overline{x}}}
\newcommand{\ybar}{{\overline{y}}}

\newcommand{\veps}{\varepsilon}

\DeclareMathOperator*{\Limsup}{Lim\,sup}

\usepackage{todonotes}

\title{Characterizations of Super-regularity and its Variants}
\author{Aris Daniilidis\thanks{DIM-CMM, Universidad de Chile, 8370459 Chile.
                              E-Mail:~\href{mailto:arisd@dim.uchile.cl}{arisd@dim.uchile.cl}}
	   \and
	   D. Russell Luke\thanks{Institut f\"ur Numerische und Angewandte Mathematik,
							  Universit\"at G\"ottingen, 37083 G\"ottingen, Germany.
		                      E-Mail:~\href{mailto:r.luke@math.uni-goettingen.de}{r.luke@math.uni-goettingen.de}}
	   \and
	   Matthew K. Tam\thanks{Institut f\"ur Numerische und Angewandte Mathematik,
	   	                    Universit\"at G\"ottingen, 37083 G\"ottingen, Germany.
	   	                    E-Mail:~{\href{mailto:m.tam@math.uni-goettingen.de}{m.tam@math.uni-goettingen.de}}}
}

\begin{document}
	
\maketitle	

\begin{abstract}
Convergence of projection-based methods for nonconvex set feasibility problems has been established for sets with ever weaker regularity assumptions. What has not kept pace with these developments is analogous results for convergence of optimization problems with correspondingly weak assumptions on the value functions.  Indeed, one of the earliest classes of nonconvex sets for which convergence results were obtainable, the class of so-called {\em super-regular sets} \cite{LewisLukeMalick09}, has no functional counterpart. In this work, we amend this gap in the theory by establishing the equivalence between a property slightly stronger than super-regularity, which we call \emph{Clarke super-regularity}, and \emph{subsmootheness} of sets as introduced by Aussel, Daniilidis and Thibault \cite{AusDanThi04}. The bridge to functions shows that approximately convex functions studied by Ngai, Luc and Thera \cite{NgaLucThe00} are those which have Clarke super-regular epigraphs. Further classes of regularity of 
functions based on the corresponding regularity of their epigraph are also discussed.
\end{abstract}


\medskip


\section{Introduction}

The notion of a \emph{super-regular} set was introduced by Lewis, Luke and Malick \cite{LewisLukeMalick09} who 
recognized the property as an important ingredient for proving convergence of the method of alternating projections without convexity. 
This was  generalized in subsequent publications \cite{HesLuk13,BauLukePhanWang13a,DaoPhan17, LukNguTam17}, with the weakest known assumptions 
guaranteeing local linear convergence of the alternating projections algorithm for two-set, consistent feasibility problems 
to date found in \cite[Theorem 3.3.5]{ThaoDiss}.  The regularity assumptions on the individual sets in these subsequent 
works are vastly weaker than super-regularity, but what has not kept pace with these generalizations is their 
functional analogs.  Indeed, it appears that the notion of a {\em super-regular function} has not yet been articulated.  
In this note, we bridge this gap between super-regularity of sets and functions as well as establishing connections to other known function-regularities in the literature.
A missing link is yet another type of set regularity, what we call \emph{Clarke super-regularity}, which is a slightly stronger version of super-regularity and, as we show, this is equivalent to other existing notions of regularity.  
For a general set that is not necessarily the epigraph of a function, we establish an equivalence 
between \emph{subsmoothness} as introduced by Aussel, Daniilidis and Thibault  \cite{AusDanThi04} and Clarke super-regularity. 

\smallskip

To begin, in Section~\ref{s:normal cones and regularity} we recall different concepts of the normal cones to a set 
as well as notions of set regularity, including Clarke regularity (Definition~\ref{Def-Clarke-reg}) and (limiting) 
super-regularity (Definition~\ref{Def-super-reg}). Next, in Section~\ref{s:superregularity and subsmoothness} 
we introduce the notion of Clarke super-regularity (Definition~\ref{Def-Clarke-super}) and relate it to the 
notion of subsmoothness (Theorem~\ref{t:semi-subsmooth char}). We also provide an example illustrating 
that Clarke super-regularity at a point is a strictly weaker condition than Clarke regularity around the 
point (Example~\ref{example}).  Finally, in Section~\ref{s:regularity of functions}, we provide analogous 
statements for Lipschitz continuous functions, relating the class of approximately convex functions to 
super-regularity of the epigraph. After completing this work we received a preprint \cite{Thib18} which contains 
results of this flavor, including a 
characterization of (limiting) super-regularity in terms of (metric) subsmoothness.

\section{Normal cones and Clarke regularity}\label{s:normal cones and regularity}

The notation used throughout this work is standard for the field of variational
analysis, as can be found in \cite{VA}. The \emph{closed ball} of radius $r>0$
centered at $x\in\mathbb{R}^{n}$ is denoted ${\mathbb{B}}_{r}(x)$ and the
closed unit ball is denoted $\mathbb{B}:=\mathbb{B}_{1}(0)$. The 
\emph{(metric) projector} onto a set $\Omega\subset\mathbb{R}^{n}$, denoted by
$P_{\Omega}:\mathbb{R}^{n}\rightrightarrows\Omega$, is the multi-valued
mapping defined by
\[
P_{\Omega}(x):=\{\omega\in\Omega:\Vert x-\omega\Vert=d(x,\Omega)\},
\]
where $d(x,\Omega)$ denotes the distance of the point $x\in\mathbb{R}^n$ to the set $\Omega$. 
When $\Omega$ is nonempty and closed, its projector $P_{\Omega}$ is everywhere
nonempty.  A selection from the projector is called a {\em projection}.\smallskip

Given a set $\Omega$, we denote its \emph{closure} by $\operatorname{cl}
\Omega$, its \emph{convex hull} by $\operatorname{conv}\Omega$, and its
\emph{conic hull} by $\operatorname{cone}\Omega$. In this work we shall deal with 
two fundamental tools in nonsmooth analysis; \emph{normal cones} to sets
and \emph{subdifferentials} of functions (Section~\ref{s:regularity of functions}).

\begin{definition}
[normal cones]Let $\Omega\subseteq\mathbb{R}^{n}$ and let $\bar{\omega}
\in\Omega$.

\begin{enumerate}
[(i)]

\item The \emph{proximal normal cone} of $\Omega$ at $\bar{\omega}\in\Omega$
is defined by
\[
N_{\Omega}^{\mathrm{P}}(\bar{\omega})=\operatorname{cone}\left(  P_{\Omega
}^{-1}\bar{\omega}-\bar{\omega}\right)  .
\]
Equivalently, $\bar{\omega}^{\ast}\in N_{\Omega}^{\mathrm{P}}(\bar{\omega})$
whenever there exists $\sigma\geq0$ such that
\[
\langle\bar{\omega}^{\ast},\omega-\bar{\omega}\rangle\leq\sigma\Vert
\omega-\bar{\omega}\Vert^{2},\quad\forall\omega\in\Omega.
\]

\item The \emph{Fr\'echet normal cone} of $\Omega$ at $\bar{\omega}$ is defined
by
\[
\hat{N}_{\Omega}(\bar{\omega})=\left\{  \bar{\omega}^{\ast}\in\mathbb{R}
^{n}:\,\langle\bar{\omega}^{\ast},\omega-\bar{\omega}\rangle\leq
\mathrm{o}(\Vert\omega-\bar{\omega}\Vert),\,\forall\omega\in\Omega\right\},
\]
Equivalently, $\bar{\omega}^{\ast}\in\hat{N}_{\Omega}(\bar{\omega}),$ if for
every $\varepsilon>0$ there exists $\delta>0$ such that
\begin{equation}
\langle\bar{\omega}^{\ast},\omega-\bar{\omega}\rangle\leq\varepsilon
\Vert\omega-\bar{\omega}\Vert,\quad\text{for all }\omega\in\Omega
\cap{\mathbb{B}}_{\delta}(\bar{\omega}). \label{d0}
\end{equation}

\item The \emph{limiting normal cone} of $\Omega$ at $\bar{\omega}$ is defined
by
\[
{N}_{\Omega}(\bar{\omega})=\Limsup_{\omega\rightarrow\bar{\omega}}\hat
{N}_{\Omega}(\bar{\omega}),
\]
where the limit superior denotes the \emph{Painlev\'e--Kuratowski outer limit}.

\item The \emph{Clarke normal cone} of $\Omega$ at $\bar{\omega}$ is defined
by
\[
N_{\Omega}^{\text{\textrm{C}}}(\bar{\omega})=\operatorname{cl}
\operatorname{conv}{N}_{\Omega}(\bar{\omega}).
\]

\end{enumerate}

When $\bar{\omega}\not \in \Omega$, all of the aforementioned normal cones at $\bar{\omega}$ are
defined to be empty.
\end{definition}

Central to our subsequent analysis is the notion of a \emph{truncation} of a normal cone. 
Given $r>0,$ one defines the \emph{$r$-truncated} version of each of the above cones to be its 
intersection with a ball centered at the origin of radius $r$. For
instance, the \emph{$r$-truncated proximal normal cone} of $\Omega$ at
$\bar{\omega}\in\Omega$ is defined by
\[
N_{\Omega}^{\text{\textrm{rP}}}(\bar{\omega})=\operatorname{cone}\left(
P_{\Omega}^{-1}\bar{\omega}-\bar{\omega}\right)  \cap{\mathbb{B}}_{r},
\]
that is, $\bar{\omega}^{\ast}\in N_{\Omega}^{\text{\textrm{rP}}}(\bar{\omega
})$ whenever $\Vert\bar{\omega}^{\ast}\Vert\leq r$ and for some $\sigma\geq0$
we have
\[
\langle\bar{\omega}^{\ast},\omega-\bar{\omega}\rangle\leq\sigma\Vert
\omega-\bar{\omega}\Vert^{2},\quad\forall\omega\in\Omega.
\]

In general, the following inclusions between the normal cones can deduce straightforwardly from
their respective definitions:
\begin{equation}
N_{\Omega}^{\text{\textrm{P}}}(\bar{\omega})\subseteq{\hat{N}}_{\Omega}
(\bar{\omega})\subseteq{N}_{\Omega}(\bar{\omega})\subseteq N_{\Omega
}^{\text{\textrm{C}}}(\bar{\omega}). \label{eq:normal cone inclusions}
\end{equation}

The regularity of sets is characterized by the relation between elements in the 
graph of the normal cones to the sets and directions constructable from 
points in the sets.  The weakest kind of regularity of sets 
that has been shown to guarantee convergence of the alternating projections 
algorithm is {\em elemental subregularity} (see \cite[Cor.3.13(a)]{HesLuk13}
and \cite[Theorem 3.3.5]{ThaoDiss}).  It was called {\em elemental} (sub)regularity 
in \cite[Definition~5]{KruLukNgu16} and  \cite[Definition~3.1]{LukNguTam17} to 
distinguish regularity of sets from regularity of collections of sets.  Since we are only 
considering the regularity of sets, and later functions, we can drop the ``elemental'' 
qualifier in the present setting.  We also streamline the terminology and variations on
elemental subregularity used in \cite{KruLukNgu16, LukNguTam17}, replacing
{\em uniform} elemental subregularity  with a more versatile and easily distinguishable 
variant.  
\begin{definition}[{subregularity {\cite[Definition~5]{KruLukNgu16}}}]\label{Def-elem-subreg}Let
$\Omega\subseteq\mathbb{R}^{n}$ and $\bar{\omega}\in\Omega$. The set $\Omega$
is said to be \emph{$\epsilon$-subregular relative to $\Lambda$ at $\bar{\omega}$ 
for $(\hat\omega, \hat\omega^*)\in\gph \ncone{\Omega}$}
if it is locally closed at
$\bar{\omega}$ and there exists an $\epsilon>0$ together with a neighborhood $U$ of $\bar\omega$ 
such that 
\begin{equation}
\left\langle \hat\omega^{\ast}-(\omega'-\omega),
\omega-\hat\omega\right\rangle
\leq\varepsilon||\hat\omega^{\ast}-(\omega'-\omega)||\Vert \omega-\hat\omega\Vert\,,\quad\forall \omega'\in\Lambda\cap U,~ \forall \omega\in P_{\Omega}(\omega'). \label{e:esr}
\end{equation}
If {\em for every} $\epsilon>0$ there is a neighborhood (depending on $\epsilon$) such that \eqref{e:esr} holds, then 
$\Omega$ is said to be \emph{subregular relative to $\Lambda$ at $\bar{\omega}$ 
for $(\hat\omega, \hat\omega^*)\in\gph \ncone{\Omega}$}.
\end{definition}
The property that distinguishes the degree of regularity of sets is the diversity of vectors 
$(\hat\omega, \hat\omega^*)\in\gph \ncone{\Omega}$ for which \eqref{e:esr} holds, as well as the 
choice of the set $\Lambda$.  
Of particular interest to us are {\em Clarke regular} sets, which satisfy \eqref{e:esr} for all 
$\epsilon>0$ and for all Clarke normal vectors at $\bar\omega$.    
\begin{definition}
[{Clarke regularity}]\label{Def-Clarke-reg}The set $\Omega$
is said to be \emph{Clarke regular} at $\bar{\omega}\in\Omega$ if it is locally closed at
$\bar{\omega}$ and for every  $\varepsilon>0$ there exists $\delta>0$
such that for all $(\bar\omega, \bar{\omega}^*)\in\gph\cncone{\Omega}$
\begin{equation}
\left\langle \bar{\omega}^{\ast},
\omega-\bar{\omega}\right\rangle
\leq\varepsilon\,||\bar{\omega}^{\ast}||\Vert\omega-\bar{\omega}\Vert, \quad\forall \omega\in\Omega\cap{\mathbb{B}}_{\delta}(\bar{\omega}). \label{d1}
\end{equation}
\end{definition}
Note that \eqref{d1} is \eqref{e:esr} with $\Lambda=\Omega$ and $U=\mathbb{B}_\delta(\bar\omega)$, which in 
the case of Clarke regularity holds for all $(\bar\omega, \bar{\omega}^*)\in\gph\cncone{\Omega}$.  
A short argument shows that, for $\Omega$ Clarke regular at $\bar\omega$,  the Clarke and Fr\'echet normal 
cones coincide at $\bar\omega$.  Indeed, this property is used to {\em define} Clarke regularity in 
{\cite[Definition~6.4]{VA}}.  It is also immediately clear from the definitions that if $\Omega$ is  
Clarke regular at $\bar\omega$, then it is subregular relative to $\Lambda=\Omega$ 
at $\bar\omega$ for all $\bar\omega^*\in\ncone{\Omega}(\bar\omega)$.  

By setting $\Lambda=\mathbb{R}^n$, letting $\hat\omega \in \Omega$ be in a neighborhood of $\bar \omega$ and 
fixing $\hat{\omega}^*=0$ in the context of Definition~\ref{Def-elem-subreg}, we arrive at super-regularity which, 
when stated explicitly, takes the following form.
\begin{definition}[{super-regularity {\cite[Definition~4.3]{LewisLukeMalick09}}}]\label{Def-super-reg}
Let $\Omega\subseteq\mathbb{R}^{n}$ and $\bar{\omega}\in\Omega$. The set $\Omega$
is said to be \emph{super-regular at $\bar{\omega}$}
if it is locally closed at
$\bar{\omega}$ and for every $\epsilon>0$ there is a $\delta>0$   
such that for all $(\hat\omega, 0)\in\gph\ncone{\Omega}\cap \left\{\left(\mathbb{B}_\delta(\bar\omega), 0\right)\right\}$
\begin{equation}
\left\langle \omega'-\omega,
\hat\omega-\omega\right\rangle
\leq\varepsilon\,||\omega'-\omega||\Vert \hat\omega-\omega\Vert,\quad\forall \omega'\in  \mathbb{B}_\delta(\bar\omega),~ \forall \omega\in P_{\Omega}(\omega'). \label{e:sup-reg}
\end{equation}
\end{definition}
\noindent Rewriting the above leads the the following equivalent characterization of super-regularity, which is more 
useful for our purposes.
\begin{proposition}[\hspace{-0.05ex}{\cite[Proposition~4.4]{LewisLukeMalick09}}]\label{t:super-reg}The set $\Omega\subseteq\mathbb{R}^n$
is \emph{super-regular} at $\bar{\omega}\in\Omega$ if and only if it is locally closed at
$\bar{\omega}$ and for every  $\varepsilon>0$ there exists $\delta>0$
such that
\begin{eqnarray}
&&\left\langle \omega_1^{\ast},
\omega_2-\omega_1\right\rangle
\leq\varepsilon\,||\omega_1^{\ast}||\Vert\omega_2-\omega_1\Vert, \nonumber\\
&&\qquad\qquad\qquad\qquad\forall (\omega_1,\omega_1^{\ast})\in 
\gph\ncone{\Omega}\cap\left(\mathbb{B}_{\delta}(\bar{\omega})\times \mathbb{R}^n\right),\quad 
\forall \omega_2\in\Omega\cap\mathbb{B}_{\delta}(\bar{\omega}). \label{e:sup-reg2}
\end{eqnarray}
\end{proposition}
\noindent It is immediately clear from this characterization that super-regularity implies Clarke regularity. 
By continuing our development of increasingly nicer regularity properties to convexity, we have the following 
relationships involving stronger notions of regularity.  
\begin{proposition}\label{t:relationships}
 Let $\Omega\subseteq\mathbb{R}^n$ be locally closed at $\bar\omega\in\Omega$. 
 \begin{enumerate}[(i)]
  \item\label{t:relationships i} If $\Omega$ is prox-regular at $\bar\omega$ (\emph{i.e.,} there exists a neighborhood of $\xbar$ on which the projector is single-valued), then there is a constant $\gamma>0$ such that for all $\epsilon>0$ 
 \begin{eqnarray}
&&\left\langle \omega_1^{\ast},\omega_2-\omega_1\right\rangle
\leq\varepsilon||\omega_1^{\ast}||\Vert \omega_2-\omega_1\Vert\,,
\nonumber\\
&&\qquad\qquad\qquad
\forall (\omega_1, \omega_1^*) \in \gph\ncone{\Omega}\cap\left(\mathbb{B}_{\gamma\epsilon}(\bar\omega)\times\mathbb{R}^n\right), 
~ \forall \omega_2\in\Omega\cap \mathbb{B}_{\gamma\epsilon}(\bar\omega). \label{e:pr}
\end{eqnarray}
\item\label{t:relationships ii} If $\Omega$ is convex, then
 \begin{equation}
\left\langle \omega_1^{\ast},\omega_2-\omega_1\right\rangle
\leq 0\,,\qquad  \forall(\omega_1, \omega_1^*) \in \gph\ncone{\Omega},
\quad \forall \omega_2\in\Omega. \label{e:cvx}
\end{equation}
 \end{enumerate}
\end{proposition}
\begin{proof}
The proof of \eqref{t:relationships i} can be found in \cite[Proposition 4(vi)]{KruLukNgu16}.  Part \eqref{t:relationships ii}
is classical. 
\end{proof}

\begin{example}[Pac-Man]\label{eg:pacman}
Let $\xbar = 0\in\mathbb{R}^2$ and consider two subsets of $\mathbb{R}^2$ given by
 \begin{align*}
  A&=\{ (x_1,x_2)\in\mathbb{R}^2~|~ x_1^2+ x_2^2\leq 1, ~-(1/2)x_1\leq x_2\leq x_1, x_1\geq 0\}, \\
  B&=\{ (x_1,x_2)\in\mathbb{R}^2~|~ x_1^2+ x_2^2\leq 1, ~x_1\leq |x_2|\}.
 \end{align*}
The set $ B$ looks like a ``Pac-Man''' with mouth opened to the right and the set $A$, if you like,  
a piece of pizza. For an illustration, see Figure~\ref{f:pacman}.  The set $B$ is subregular relative to $A$ at $\xbar = 0$ for all $(b,v)\in \gph{\paren{\ncone{B}\cap A}}$ for $\epsilon=0$ 
on all neighborhoods since, for all 
$a\in A$, $a_B\in P_{B}(a)$ and $v\in \ncone{B}(b)\cap A$. To see this, we simply note that
\[
 \ip{v-(a-a_B)}{a_B-b} = \ip{v}{a_B-b}-\ip{a-a_B}{a_B-b} = 0.
\]
In other words, from the perspective of the piece of pizza, Pac-Man looks {\em convex}.
The set $B$, however, is only $\epsilon$-subregular at $\xbar=0$ relative to $\mathbb{R}^2$ 
for any $v\in\ncone{B}(0)$ for $\epsilon=1$ 
since, by choosing  $x=tv \in  B$ (where $0\neq v\in  B\cap \ncone{B}(0)$, $t\downarrow 0$),
we have $\ip{v}{x}=\|v\|\|x\|>0$.  Clearly, this also means that Pac-Man is not Clarke regular.  
\end{example}

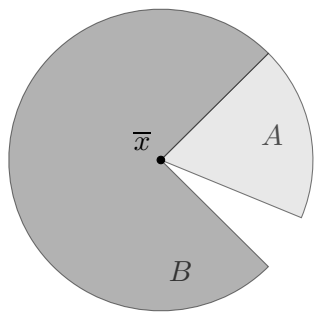
\begin{figure}[!htb]
\centering
\begin{tikzpicture}
  \draw[draw=black!80!gray,fill=gray!30,opacity=0.6] (0,0) -- (-22.5:2cm) arc (-22.5:45:2cm) -- cycle;
  \draw[draw=black!80!gray,fill=gray,opacity=0.6] (0,0) -- (45:2cm) arc (45:315:2cm) -- cycle;
  \draw[fill=black] (0,0) circle (0.3ex) node[above left] {$\xbar$};
  \draw[color=black!40!gray] (12.5:1.5cm) node {$A$};
  \draw[color=black!60!gray] (-80:1.5cm) node {$B$};
\end{tikzpicture}
\caption{An illustration of the sets in Example~\ref{eg:pacman}.\label{f:pacman}}
\end{figure}

\section{Super-regularity and subsmoothness}\label{s:superregularity and subsmoothness}

In the context of the definitions surveyed in the previous section,  
we introduce an even stronger type of regularity that we identify, in Theorem \ref{t:semi-subsmooth char}, 
with {\em subsmoothness} as studied in \cite{AusDanThi04}. This will provide a crucial link to the analogous 
characterizations of regularity for {\em functions} considered in Theorem \ref{th:approx convex iff CSR}, in 
particular, to {\em approximately convex} functions studied in \cite{NgaLucThe00}.
\begin{definition}
[Clarke super-regularity]\label{Def-Clarke-super}Let $\Omega\subseteq\mathbb{R}^{n}$ and 
$\bar{\omega}\in\Omega$. The set $\Omega$ is said to be \emph{Clarke
super-regular} at $\bar{\omega}$ if it is locally closed at $\bar{\omega}$ and
for every $\varepsilon>0$ there exists $\delta>0$ such that for every
$(\hat\omega,\hat\omega^{\ast})\in \gph\cncone{\Omega}\cap\left(\mathbb{B}_{\delta}(\bar{\omega})\times \mathbb{R}^n\right)$, the following inequality holds
\begin{equation}
\left\langle \hat\omega^{\ast},
\omega-\hat\omega\right\rangle \leq
\varepsilon\,||\hat\omega^{\ast}||\Vert\omega-\hat\omega\Vert, \quad\forall \omega\in\Omega\cap\mathbb{B}_{\delta}(\bar{\omega}).\label{e:Clarke-super}
\end{equation}
\end{definition}
The only difference between Clarke super-regularity and super-regularity is that, in the case of 
Clarke super-regularity, the key inequality above holds for all nonzero Clarke normals in a neighborhood 
instead holding only for limiting normals (compare \eqref{e:sup-reg2} with~\eqref{e:Clarke-super}). 
It therefore follows that Clarke super-regularity at a point implies Clarke regularity there.  
Nevertheless, even this stronger notion of regularity does not imply Clarke regularity around $\bar{\omega}$, 
as the following counterexample shows.

\begin{figure}[htb]
	\centering
	\begin{tikzpicture}
	\draw[<-] (-0.7,0) -- (0,0);
	\draw[densely dotted] (0,0) -- (0.4,0);
	\draw (0.4,0) -- (-5:1) -- (-12.5:2.5) -- (-25:4.5);
	\draw[densely dotted] (-25:4.5) -- ($ (-25:4.5) + (0.4,-0.1) $);
	\draw[->] ($ (-25:4.5) + (0.4,0-0.1) $) -- ($ (-25:4.5) + (1.1,-0.1) $);
	\draw[draw=none,fill=gray!50,opacity=0.5] (-0.7,1.5) -- (-0.7,0) -- (0,0) -- (0.4,0)  -- (-5:1) -- (-12.5:2.5) -- (-25:4.5) -- ($ (-25:4.5) + (0.4,-0.1) $) -- ($ (-25:4.5) + (0.4,0-0.1) $) -- ($ (-25:4.5) + (1.1,-0.1) $) -- ($ (-25:4.5) + (1.1,3.375) $) -- cycle;
	\draw[color=gray!60!black] ($ (-25:4.5) + (1.1,3.375) $)  node[below left] {$\epi f$};
	\draw[fill] (0,0) circle (0.3ex) node[below] {$\bar{\omega}$};
	\draw[fill] (-5:1) circle (0.3ex) node[below] {$\omega_{k+1}$};
	\draw[fill] (-12.5:2.5) circle (0.3ex) node[below] {$\omega_{k}$};    
	\draw[fill] (-25:4.5) circle (0.3ex) node[below] {$\omega_{k-1}$};        
	\draw[<->] (-0.8,1.6) -- (-0.8,-2.4) -- ($ (-25:4.5) + (1.2,-0.5) $);
	\draw ($ (-25:4.5) + (1.1,-0.5) $) node[below left] {$x$};
	\draw  (-0.8,1.6) node [below left] {$f(x)$};
	\end{tikzpicture}
	\caption{A sketch of the function $f$ and the sequence $(\omega_k)$ given in Example~\ref{example}.\label{f:counterexample}}
\end{figure}
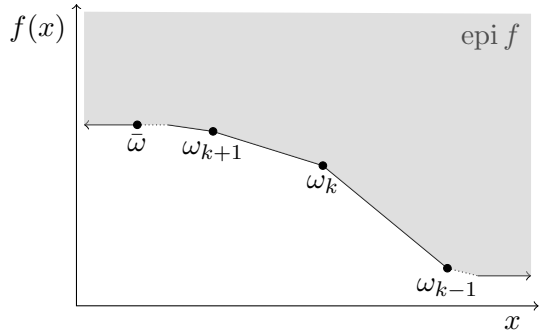

\begin{example}[regularity only at a point]\label{example}
Let $f:\mathbb{R}^{2}\rightarrow\mathbb{R}$ be the continuous, piecewise linear function (see Figure~\ref{f:counterexample}) defined by
\[
f(x):=
\begin{cases}
\phantom{arismat} 0, & \text{if \,\,} x\leq0 \smallskip \\
-\frac{1}{2^{k}}(x-\frac{1}{2^{k}})-\frac{1}{3\cdot4^{k}}\,, & \text{if \,\,} \frac{1}{2^{k+1}
}\leq x\leq\frac{1}{2^{k}}\quad\text{(for }k=1,2,\ldots\text{)} \smallskip \\
\phantom{mat}-\frac{1}{12}, & \text{if \,\,} x\geq\frac{1}{2}\,.
\end{cases}
\]
Notice that
\begin{equation}
-\frac{4}{3}x^{2}\leq f(x)\leq-\frac{1}{3}x^{2},\quad\forall 
x\in\left[0,\frac{1}{2}\right]. \label{d3}
\end{equation}
Let $\Omega=\operatorname{epi}f$ denote the epigraph of $f$. Thanks to
(\ref{d3}) it is easily seen that $\Omega$ is Clarke regular at $\bar{\omega
}=(0,0)$ in the sense of Definition \ref{Def-Clarke-reg}. However, $\Omega$ is
not Clarke regular at the sequence of points $\omega_{k}=(\frac{1}{2^{k+1}
},\frac{1}{2^{k}})$ converging to $\bar{\omega}.$ Indeed, the Fr\'echet normal cones 
$\hat{N}_{\Omega}(\omega_{k})$ are reduced to $\{0\}$ for all $k\ge 1$, while
the corresponding limiting normal cones are given by
\[
{N}_{\Omega}(\omega_{k})=\mathbb{R}_{+}\left\{  \left(-\frac{1}{2^{k}},-1\right),\left(-\frac{1}{2^{k+1}},-1\right)\right\}  ,\quad\forall k\in\mathbb{N}.\vspace{-1em}
\]
\end{example}

A missing link in the cascade of set regularity is subsmooth and semi-subsmooth sets, 
introduced and studied 
by Aussel, Daniilidis and Thibault in \cite[Definitions~3.1~\&~3.2]{AusDanThi04}.
\begin{definition}
[(Semi-)subsmooth sets]\label{d:(semi-)subsmooth} Let $\Omega\subset
\mathbb{R}^{n}$ be closed and let $\bar{\omega}\in\Omega$.
\begin{enumerate}[{\rm (i)}]
\item The set $\Omega$ is \emph{subsmooth} at $\bar{\omega}\in\Omega$ if, for every
$r>0$ and $\varepsilon>0$, there exists $\delta>0$ such that for all $\omega_{1},\omega_{2}\in\mathbb{B}_{\delta}(\bar{\omega})\cap\Omega$, all $\omega_{1}^{\ast}\in N_{\Omega}^{\text{\textrm{rC}}}(\omega
_{1})$ and all $\omega_{2}^{\ast}\in N_{\Omega}^{\text{\textrm{rC}}}(\omega
_{2})$ we have:
\begin{equation}
\langle\omega_{1}^{\ast}-\omega_{2}^{\ast},\omega_{1}-\omega_{2}\rangle
\geq-\varepsilon\Vert\omega_{1}-\omega_{2}\Vert.\label{eq:subsmooth}
\end{equation}

\item The set $\Omega$ is \emph{semi-subsmooth} at $\bar{\omega}$ if, for
every $r>0$ and $\varepsilon>0$, there exists $\delta>0$ such that for all
$\omega\in\mathbb{B}_{\delta}(\bar{\omega})\cap\Omega$, all $\omega^{\ast}\in N_{\Omega}^{\text{\textrm{rC}}}(\omega)$ and all $\bar{\omega}^{\ast}\in N_{\Omega}^{\text{\textrm{rC}}}({\bar{\omega}})$

\begin{equation}
\langle\omega^{\ast}-\bar{\omega}^{\ast},\omega-\bar{\omega}\rangle
\geq-\varepsilon\Vert\omega-\bar{\omega}\Vert.\label{eq:semi-subsmooth}
\end{equation}

\end{enumerate}
\end{definition}

\noindent It is clear from the definitions that subsmoothness at a point implies
semi-subsmoothness at the same point. The next theorem establishes the precise connection between subsmoothness and
Clarke super-regularity (Definition~\ref{Def-Clarke-super}).

\begin{theorem}
[characterization of subsmoothness]\label{t:semi-subsmooth char}
Let $\Omega\subseteq\mathbb{R}^{n}$ be closed and nonempty.

\begin{enumerate}[{\rm (i)}]
\item \label{th:superregular/subsmooth} The set $\Omega$ is subsmooth at
$\bar{\omega}\in\Omega$ if and only if $\Omega$ is Clarke super-regular at
$\bar{\omega}$.

\item \label{th:supersubreg/semismooth} The set $\Omega$ is semi-subsmooth at
$\bar{\omega}\in\Omega$ if and only if for each constant $\varepsilon>0$ there
is a $\delta>0$ such that for every 
$(\bar\omega, \bar{\omega}^{\ast})\in \gph \cncone{\Omega}$ 
\[
\left\langle \bar{\omega}^{\ast},\omega-\bar{\omega}\right\rangle
\leq\varepsilon\,||\bar{\omega}^{\ast}||\Vert\omega-\bar{\omega}\Vert,\quad \forall \omega\in \Omega\cap{\mathbb{B}}_{\delta}(\bar{\omega})
\]
and for all $(\omega,\omega^{\ast})\in\gph\cncone{\Omega}\cap\left({\mathbb{B}}_{\delta}(\bar{\omega})\times \mathbb{R}^n\right)$, 
\[
\left\langle \omega^{\ast},\bar{\omega}-\omega\right\rangle
\leq\varepsilon\,||\omega^{\ast}||\Vert \bar{\omega} - \omega\Vert.
\]
\end{enumerate}
\end{theorem}

\begin{proof} (i). Assume $\Omega$ is subsmooth at $\bar{\omega}\in\Omega$ and fix an
$\varepsilon>0.$ Set $r=1$ and let $\delta>0$ be given by the definition of
subsmoothness.  Then for every $\omega_{1},\omega_{2}\in\Omega\cap{\mathbb{B}%
}_{\delta}(\bar{\omega})$ and $\omega_{2}^{\ast}\in N_{\Omega}%
^{\text{\textrm{C}}}(\omega_{2})\diagdown\{0\}$, applying (\ref{eq:subsmooth})
for $\omega_{1}^{\ast}=\{0\}\in N_{\Omega}^{\text{(\textrm{r=1)C}}}(\omega
_{1})$ and $||\omega_{2}^{\ast}||^{-1}\omega_{2}^{\ast}\in N_{\Omega
}^{\text{(\textrm{r=1)C}}}(\omega_{2})$ we deduce (\ref{e:Clarke-super}).
The same argument applies in the case that $\omega_{2}^{\ast}=0$ and 
$\omega_{1}^{\ast}\neq 0$.  If both $\omega_{1}^{\ast}=\omega_{2}^{\ast}=0$, then 
the required inequality holds trivially.
\smallskip

Let us now assume that $\Omega$ is Clarke super-regular at $\bar{\omega}$ and
fix $r>0$ and $\varepsilon>0.$ Let $\delta>0$ be given by the definition of
Clarke super-regularity corresponding to $\varepsilon'=\varepsilon/2r$ and
let $\omega_{1},\omega_{2}\in\mathbb{B}_{\delta}(\bar{\omega})\cap\Omega$,
$\omega_{1}^{\ast}\in N_{\Omega}^{\text{\textrm{rC}}}(\omega_{1})$ and
$\omega_{2}^{\ast}\in N_{\Omega}^{\text{\textrm{rC}}}(\omega_{2}).$ It follows
from (\ref{e:Clarke-super}) that
\begin{eqnarray*}
\left\langle \omega_{1}^{\ast},\omega_{1}-\omega_{2}\right\rangle &\geq&-\frac{\varepsilon}{2r}||\omega
_{1}^{\ast}||\Vert\omega_{1}-\omega_{2}\Vert\geq-\frac{\varepsilon}{2}\Vert\omega_{1}-\omega_{2}\Vert\\
\quad\text{and\quad}&&\\
\left\langle-\omega_{2}^{\ast},\omega_{1}-\omega_{2}\right\rangle &\geq&
-\frac{\varepsilon}{2r}||\omega_{2}^{\ast}||\Vert\omega_{1}-\omega_{2}\Vert\geq-\frac{\varepsilon}{2}\Vert\omega_{1}-\omega_{2}\Vert.
\end{eqnarray*}
We conclude by adding the above inequalities. \medskip

Part (ii) is nearly identical and the proof is omitted. 
\end{proof}

The following corollary utilizes Theorem~\ref{t:semi-subsmooth char} to summarize the relations between 
various notions of regularity for sets, the weakest of these being the weakest known regularity assumption 
under which local convergence of alternating projections has been established 
\cite[Theorem 3.3.5]{ThaoDiss}.  
\begin{corollary}\label{cor:summary set regularies}
Let $\Omega\subseteq\mathbb{R}^n$ be closed, let $\bar{\omega}\in\Omega$ and consider the following assertions.
\begin{enumerate}[(i)]
	\item\label{it:summary set regularities ii}  $\Omega$ is prox-regular at $\bar{\omega}$.
	\item\label{it:summary set regularities iii} $\Omega$ is subsmooth at $\bar{\omega}$.
	\item\label{it:summary set regularities iv} $\Omega$ is Clarke super-regular at $\omega$.
	\item\label{it:summary set regularities v}  $\Omega$ is (limiting) super-regular at $\omega$.
	\item\label{it:summary set regularities vi}  $\Omega$ is Clarke regular at $\omega$.
	\item\label{it:summary set regularities vii}  $\Omega$ is subregular at $\omega$ relative 
	to some nonempty $\Lambda\subset\mathbb{R}^n$ for all $(\omega, \omega^*)\in V\subset\gph\pncone{\Omega}$.
\end{enumerate}
Then \eqref{it:summary set regularities ii} $\implies$ \eqref{it:summary set regularities iii} $\iff$ \eqref{it:summary set regularities iv} $\implies$ \eqref{it:summary set regularities v}$\implies$\eqref{it:summary set regularities vi}$\implies$\eqref{it:summary set regularities vii}.
\end{corollary}
\begin{proof}
\eqref{it:summary set regularities ii} $\implies$ \eqref{it:summary set regularities iii}:~This was shown in  \cite[Proposition~3.4(ii)]{AusDanThi04}.
\eqref{it:summary set regularities iii} $\iff$ \eqref{it:summary set regularities iv}:~ This is 
Theorem~\ref{t:semi-subsmooth char}\eqref{th:superregular/subsmooth}.
\eqref{it:summary set regularities iv}
$\implies$ \eqref{it:summary set regularities v}$\implies$ \eqref{it:summary set regularities vi}$\implies$ \eqref{it:summary set regularities vii}:~These implications follow from the definitions.
\end{proof}
\begin{remark}[amenablility]\label{r:amenability}
A further regularity between convexity and prox-regularity is {\em amenability} \cite[Definition 10.23]{VA}. 
This was shown in \cite[Corollary~2.12]{PolRock96a} to imply prox-regularity.  Amenability plays a larger 
role in the analysis of functions and is defined precisely in this context below. 
\end{remark}

\section{Regularity of functions}\label{s:regularity of functions}

The extension of the above notions of set regularity to analogous notions for functions typically passes through the epigraphs. Given a function
$f:\mathbb{R}^{n}\rightarrow\lbrack-\infty,+\infty]$, recall that its \emph{domain} is
$\operatorname{dom}f:=\{x\in\mathbb{R}^{n}:f(x)<+\infty\}$ and its
\emph{epigraph} is 
\[
\operatorname{epi}f:=\{(x,\alpha)\in\mathbb{R}^{n}\times\mathbb{R}
:f(x)\leq\alpha\}.
\]
The {\em subdifferential} of a function at a point $\bar{x}$ can be defined in terms of the normal cone to 
its epigraph at that point.
Let $f:\mathbb{R}^{n}\to(-\infty,+\infty]$
and let $\bar{x}\in\operatorname{dom} f$.
The \emph{proximal subdifferential} of $f$ at $\bar{x}$ is defined by
\begin{equation}
\partial^{\mathrm{P}}\!f(\bar{x})=\{v\in\mathbb{R}^{n}:(v,-1)\in
N_{\operatorname{epi}f}^{\text{\textrm{P}}}((\bar{x},f(\bar{x}))\}.
\label{eq:prox subdifferential}
\end{equation}
The \emph{Fr\'echet (resp. limiting, Clarke) subdifferential}, denoted
${\hat{\partial}}f(\bar{x})$ (resp.\ $\partial f(\bar{x})$, $\partial
^{\mathrm{C}}f(\bar{x})$), is defined analogously by replacing normal cone
$N_{\operatorname{epi}f}^{\text{\textrm{P}}}(\bar{\omega})$ with ${\hat{N}
}_{\operatorname{epi}f}(\bar{\omega})$ (resp. ${N}_{\operatorname{epi}f}
(\bar{\omega}),$\ $N_{\operatorname{epi}f}^{\text{\textrm{C}}}(\bar{\omega})$)
in \eqref{eq:prox subdifferential} where $\bar{\omega}=(\bar{x},f(\bar{x}))$.
The \emph{horizon and Clarke horizon subdifferentials} at ${\overline
{x}}$ are defined, respectively, by
\begin{align*}
{\partial}_{\infty}f(\bar{x})  &  =\{v\in\mathbb{R}^{n}:(v,0)\in
{N}_{\operatorname{epi}f}((\bar{x},f(\bar{x}))\},\\
\partial^{\mathrm{C}}\!_{\infty}f(\bar{x})  &  =\{v\in\mathbb{R}^{n}:(v,0)\in
N_{\operatorname{epi}f}^{\text{\textrm{C}}}((\bar{x},f(\bar{x}))\}.
\end{align*}

In what follows, we define regularity of functions in terms of the regularity of their epigraphs. We refer to a regularity defined in such a way as {\em epi-regularity}.  
\begin{definition}[epi-regular functions]\label{d:subregular fns}
	Let $f:\mathbb{R}^n\to(-\infty,+\infty]$, $\bar{x}\in\dom f$,  $\Lambda\subseteq\dom f$, and 
        $(\ybar, \vbar)\in \gph \lsd f \cup \gph \lsd_{\infty} f$. 
	\begin{enumerate}[(i)]
	 \item $f$ is said to be \emph{$\varepsilon$-epi-subregular} at $\bar{x}\in\dom f$ 
	 relative to $\Lambda\subseteq\dom f$ 
	 for $(\ybar, \vbar)$ whenever $\epi f$ is \emph{$\varepsilon$-subregular} at $\bar{x}\in\dom f$ 
	 relative to $\{(x,\alpha)\in\epi f~|~x\in \Lambda\}$ 
	 for $(\ybar, (\vbar,e))$ with fixed $e\in\{-1,0\}$.
	\item $f$ is said to be \emph{epi-subregular} at $\bar{x}$ relative to $\Lambda\subseteq\dom f$ 
	for $(\ybar, \vbar)$ whenever $\epi f$ is {\em subregular} at $(\bar{x}, f(\xbar))$ relative to 
	$\{(x,\alpha)\in\epi f~|~x\in \Lambda\}$ 
	for $(\ybar, (\vbar, e))$ with fixed $e\in\{-1,0\}$.
	\item $f$ is said to be  {\em epi-Clarke regular} at $\bar{x}$ 
	whenever $\epi f$ is Clarke regular at $(\bar{x}, f(\xbar))$.  Similarly, 
        the function is said to be {\em epi-Clarke super-regular} (resp.\ epi-super-regular,  epi-prox-regular) 
        at $\bar{x}$ whenever its epigraph is 
        Clarke super-regular (resp.\ super-regular, or prox-regular)  at $(\bar{x}, f(\xbar))$.
	\end{enumerate}
\end{definition}

Recent work \cite{BauBolTeb16, BSTV2018} makes use of the {\em directional} 
regularity (in particular Lipschitz regularity) of functions or their gradients.  The next example illustrates how this fits
naturally into our framework. 
\begin{example}\label{eg:neg abs}
        The negative absolute value function $f(x)=-|x|$ is the classroom example of a function that is 
	not Clarke regular at $x=0$.  It is, however, $\epsilon$-epi-subregular relative to $\mathbb{R}$ at $x=0$ for all limiting 
	subdifferentials there for the same reason that the Pac-Man of Example \ref{eg:pacman} is 
	$\epsilon$-subregular relative to $\mathbb{R}^2$ at the origin for $\epsilon=1$.  
	Indeed,  $\partial f(0)=\{-1, +1\}$ and at
	any point $(x,y)$ below $\epi f$ the vector  
	$(x,y)-P_{\epi f}(x,y) \in \{\alpha(-1, -1), \alpha(1, -1)\}$ with $\alpha\geq 0$.  
        So by the Cauchy-Schwarz inequality
	\begin{eqnarray}
	&&\left\langle (\pm1, -1)-\alpha(\pm1, -1),
	P_{\epi f}(x,y)\right\rangle \nonumber\\
	&&\quad \leq||(\pm1, -1)-\alpha(\pm1, -1)||\Vert P_{\epi f}(x,y)\Vert\,,
	\quad\forall (x,y)\in\mathbb{R}^2.
	\end{eqnarray}
	In particular, any point $(x,x)\in\gph f$ we have 
        \[ P_{\epi f}(x,x)=(x,x)\quad \mbox{ and }\quad 
	(x,x)-P_{\epi f}(x,x) =(0,0),\] 
        so the inequality is tight for the subgradient $-1\in\partial f(0)$. 
	Following \eqref{e:esr},  this shows that $\epi f$ is $\epsilon$-subregular at the origin relative to 
	$\mathbb{R}^2$ for all limiting normals (in fact, for all Clarke normals) at $(0,0)$ for $\epsilon=1$.    
	In contrast, the function $f$ is not epi-subregular at $x=0$
	relative to $\mathbb{R}$ since the inequality above is tight on all balls around the origin, just as with  
	the Pac-Man of Example \ref{eg:pacman}.  If one 
	employs the restriction $\Lambda=\{x~|~x<0\}$ then epi-subregularity of $f$ is recovered at the origin
	relative to the negative orthant  for the subgradient $v=1$ for $\epsilon=0$ on the 
	neighborhood $U=\mathbb{R}$, that is,  $-|x|$ looks {\em convex} from this direction!  
\end{example}

In a subsequent section, we develop an equivalent, though more elementary, characterizations of these  
regularities of functions defined in Definition~\ref{d:subregular fns}.

\subsection{Lipschitz continuous functions}\label{s:lip}
In this section, we consider the class of locally Lipschitz functions, which allows us to avoid the horizon 
subdifferential (since this is always $\{0\}$ for Lipschitz functions). Recall that a set $\Omega$ is 
called \emph{epi-Lipschitz} at $\bar{\omega}\in\Omega$ if it 
can be represented near $\bar{\omega}$ as the epigraph of a Lipschitz continuous function. 
Such a function is called a \emph{locally Lipschitz representation} of $\Omega$ at $\bar{\omega}$.

The following notion of \emph{approximately convex} functions were introduced by 
Ngai, Luc and Thera \cite{NgaLucThe00} and turns out to fit naturally within our framework. 

\begin{definition}[approximate convexity]
A function $f:\mathbb{R}^{n}\rightarrow(-\infty
,+\infty]$ is said to be \emph{approximately convex} at $\bar{x}\in
\mathbb{R}^{n}$ if for every $\varepsilon>0$ there exists $\delta>0$ such
that
\[
\begin{gathered} (\forall x,y \in\mathbb{B}_\delta(\bar{x}))(\forall t\in\,]0,1[\,):\\ f(tx+(1-t)y)\leq tf(x)+(1-t)f(y)+\varepsilon t(1-t)\|x-y\|. \end{gathered}
\]

\end{definition}

Daniilidis and Georgiev \cite{DanGeo} and subsequently Daniilidis and Thibault \cite[Theorem~4.14]{AusDanThi04} showed the connection between 
approximately convex functions and subsmooth sets. Using our results in the previous section, we are able to provide the following extension of their characterization. In what follows, set $\omega=(x,t)\in \mathbb{R}^n\times\mathbb{R}$ and denote by $\pi({\omega})=x$ its projection onto $\mathbb{R}^n$. 

\begin{proposition}
[subsmoothness of Lipschitz epigraphs]\label{prop:subsmooth epi} Let $\Omega$ be an
epi-Lipschitz subset of $\mathbb{R}^n$ and let ${\bar{\omega}}\in\mathrm{bdry}
\Omega$. Then the following statements are equivalent: 

\begin{enumerate} [(i)] 

\item\label{it:subsmooth epi i}  $\Omega$ is Clarke super regular at ${\bar{\omega}}$. 

\item\label{it:subsmooth epi ii} $\Omega$ is subsmooth at ${\bar{\omega}}$. 

\item\label{it:subsmooth epi iii} every locally Lipschitz representation $f$ of $\Omega$ at ${\bar{\omega
}}$ is approximately convex at $\pi({\bar{\omega}})$. 

\item\label{it:subsmooth epi iv} some locally Lipschitz representation $f$ of $\Omega$ at ${\bar{\omega}}$ 
is approximately convex at $\pi({\bar{\omega}})$.
\end{enumerate}
\end{proposition}
\begin{proof}
The equivalence of \eqref{it:subsmooth epi i} and \eqref{it:subsmooth epi ii} follows from Theorem~\ref{t:semi-subsmooth char}\eqref{th:superregular/subsmooth}, and does not require $\Omega$ to be epi-Lipschitz.  The equivalence of 
\eqref{it:subsmooth epi ii}, \eqref{it:subsmooth epi iii} and \eqref{it:subsmooth epi iv} by \cite[Theorem~4.14]{AusDanThi04}.
\end{proof}
\begin{remark}
 The equivalences in Theorem  \ref{prop:subsmooth epi} actually hold in the Hilbert space setting without 
 any changes.  In fact, the equivalence of \eqref{it:subsmooth epi ii}-\eqref{it:subsmooth epi iv} remains true in
 Banach spaces \cite[Theorem~4.14]{AusDanThi04}.
\end{remark}

The following characterization extends \cite[Theorem~2]{DanGeo}.
\begin{theorem}[characterizations of aproximate convexity]\label{th:approx convex iff CSR}
	Let $f:\mathbb{R}^n\to\mathbb{R}$ be locally Lipschitz on $\mathbb{R}^n$ and  let $\bar{x}\in\mathbb{R}^n$. Then the following are equivalent.
	\begin{enumerate}[(i)]
		\item\label{it:approx convex iff CSR i}  $\epi f$ is Clarke super-regular at $\bar{x}$.
		\item\label{it:approx convex iff CSR ii} $f$ is approximately convex at $\bar{x}$.
		\item\label{it:approx convex iff CSR iii} For every $\veps>0$, there exists a $\delta>0$ such that 
		\begin{equation*}
		\begin{gathered}
		(\forall x,y\in\mathbb{B}_\delta(\bar{x}))(\forall v\in\csd \! f(x))\qquad
		f(y) - f(x) \geq \langle v,y-x\rangle-\veps\|y-x\|.
		\end{gathered}
		\end{equation*}   	 
		\item\label{it:approx convex iff CSR iv} $\partial f$ is submonotone \cite[Definition 7]{DanGeo} at $x_0$, that is, 
		for every $\epsilon$ there is a $\delta$ such that  for all $x_1, x_2\in \mathbb{B}_\delta(x_0)\cap \dom \partial f$, and 
		all $x_i^*\in\partial f(x_i)$ ($i=1,2$), one has
		\begin{equation}\label{e:submon}
		 \left\langle x_1^*-x_2^*, ~x_1-x_2 \right\rangle\geq -\epsilon\|x_1-x_2\|.
		\end{equation}
	\end{enumerate}
\end{theorem}
\begin{proof}
	\eqref{it:approx convex iff CSR i} $\iff$ \eqref{it:approx convex iff CSR ii}:~Since $f$ is locally Lipschitz at $\bar{x}$, it is trivially a local Lipschitz representation of $\Omega=\epi f$ at $\bar{\omega}=(\bar{x},f(\bar{x}))\in\Omega$. The result thus follows from Proposition~\ref{prop:subsmooth epi}.
	\eqref{it:approx convex iff CSR ii} $\iff$ \eqref{it:approx convex iff CSR iii} $\iff$ \eqref{it:approx convex iff CSR iv}:~This  is \cite[Theorem~2]{DanGeo}.
\end{proof}

\subsection{Non-Lipschitzian functions}
In this section, we collect results which hold true without assuming local Lipschitz continuity.
\begin{proposition}
Let $f:\mathbb{R}^{n}\rightarrow\mathbb{R}$ be lower semicontinuous (lsc) and
approximately convex. Then $\epi f$ is Clarke-super regular. 
\end{proposition}
\begin{proof}
As a proper, lsc, approximately convex function is locally Lipschitz at each point in the interior of its domain \cite[Proposition~3.2]{NgaLucThe00} and $\dom f=\mathbb{R}^n$, the result follows from Theorem~\ref{th:approx convex iff CSR}.
\end{proof}

\begin{example}[Clarke super-regularity does not imply approximate convexity]
 Consider the counting function $f:\mathbb{R}^n\to\{0,1,\dots,n\}$ defined by
 \[ 
f(x)=\|x\|_0:= \sum_{j=1}^n |\mbox{sign}(x_j)|,\quad\text{where}\quad
\mbox{sign}(t):=\begin{cases} -1&\mbox{ for } t<0\\
                0&\mbox{ for } t=0\\
                +1&\mbox{ for } t>0\,.
               \end{cases}
\]
This function is lower-semicontinuous and Clarke epi-super-regular almost everywhere, but not locally Lipschitz at $x$ 
whenever $\|x\|_0<n$; \emph{a fortiori}, $f$ it is actually discontinuous at all such points.
Indeed, the epigraph of $f$ is \emph{locally convex} almost everywhere and, in particular, 
at any point $(x, \alpha)$ with $\alpha> f(x)$.  At the point  $(x, f(x))$ however, the epigraph is not even 
Clarke regular when $\|x\|_0<n$.  Nevertheless, it is $\epsilon$-subregular, for the limiting subgradient $0$ with 
$\epsilon=1$.  Conversely, if $x$ is any point 
with $\|x\|_0=n$, then the counting function is locally constant and so in fact locally convex. 
These observations agree nicely with those in \cite{Le12}, namely, that the rank function (a generalizaton
of the counting function) is subdifferentially regular everywhere (\emph{i.e.,} all the various subdifferentials coincide) 
with $0\in\partial \|x\|_0$ for all $x\in\mathbb{R}^n$.  
\end{example}

In order to state the following corollary, recall that an extended real-valued function $f$ is strongly amenable at $\bar{x}$ is if $f(\bar{x})$ is finite and there exists an open neighborhood $U$ of $\bar{x}$ on which $f$ has a representation as 
a composite $g\circ F$ with $F$ of class $\mathcal{C}^2$ and $g$ a proper, lsc, convex function on $\mathbb{R}^n$.

\begin{proposition}
Let $f:\mathbb{R}^{n}\rightarrow(-\infty,+\infty]$ and consider the following assertions.
\begin{enumerate}[(i)]
	\item\label{it:strongly amenable} $f$ is strongly amenable at $\bar{x}$.
	\item\label{it:prox regular} $f$ is prox-regular at $\bar{x}$.
	\item\label{it:CSR epi} $\epi f$ is Clarke super-regular at $(\bar{x},f(\bar{x}))$.
\end{enumerate}	
Then: \eqref{it:strongly amenable} $\implies$ \eqref{it:prox regular} $\implies$ \eqref{it:CSR epi}.
\end{proposition}
\begin{proof}
The fact that strong amenability implies prox-regularity is discussed in \cite[Proposition~2.5]{PolRock96a}. To see that \eqref{it:prox regular} implies \eqref{it:CSR epi}, suppose $f$ is prox-regular at $\bar{x}$. Then $\epi f$ is prox-regular at $(\bar{x},f(\bar{x}))$ by \cite[Theorem~3.5]{PolRock96a} and hence Clarke super-regular at $(\bar{x},f(\bar{x}))$ by Theorem~\ref{t:semi-subsmooth char}.
\end{proof}

To conclude, we establish a {\em primal} characterization of epi-subregularity analogous to the characterization of Clarke epi-super-regularity in Theorem~\ref{th:approx convex iff CSR}. It is worth noting that, unlike the results in Section~\ref{s:lip}, this characterization includes the  possibly of horizon normals.
In what follows, we denote the epigraph of a function $f$ restricted to a subset $\Lambda\subset\dom f$ by 
$\epi (f_\Lambda):= \{(x,\alpha)\in\epi f~|~x\in\Lambda\}.$
\begin{proposition}\label{t:elem reg epi f}
Consider a function $f:\mathbb{R}^n\to (-\infty,+\infty]$, let $\xbar\in\dom f$ and let
$(\xbar, \vbar)\in(\gph\csd f\cup \gph\csd_{\infty} f)$. Then the following assertions hold.
\begin{enumerate}[(i)] 
 \item\label{t:elem reg epi f i} $f$ has an $\epsilon$-subregular epigraph at 
 $\bar{x}\in\dom f$ relative to $\Lambda\subseteq\dom f$ 
 for $(\xbar, \vbar)$ if and only if for some constant $\veps>0$ 
 there is a neighborhood $U$ of $(\xbar, f(\xbar))$ such that,  
for all $(x, \alpha)\in \epi (f_\Lambda)\cap U$, one of the following two inequalities holds:
\begin{subequations}
\begin{align}
f(\xbar) + \langle \vbar,x-\xbar\rangle & \leq 
	\alpha + \veps\|\vbar\|\|x-\xbar\|\left(\left(1+\|\vbar\|^{-2}  \right)\left(1+|\alpha-f(\xbar)|^2\|x-\xbar\|^{-2} \right) \right)^{1/2},\label{eq:subregular fns} \\
	\langle \vbar,x-\xbar\rangle &\leq \veps\|\vbar\|\|x-\xbar\|\left(1+|\alpha-f(\xbar)|^2\|x-\xbar\|^{-2} \right)^{\frac{1}{2}}.\label{eq:subregular horizon fns}
\end{align}	
\end{subequations}	
\item\label{t:elem reg epi f iii} $f$ is epi-subregular at 
 $\bar{x}\in\dom f$ for $(\xbar, \vbar)$ relative to $\Lambda\subseteq\dom f$  if and only if for all $\varepsilon>0$ there is a neighborhood (depending on $\varepsilon$)
of $(\xbar, f(\xbar))$ such that, for all $(x, \alpha)\in \epi (f_\Lambda)\cap U$, either \eqref{eq:subregular fns} or \eqref{eq:subregular horizon fns} holds. 
\end{enumerate}
\end{proposition}
\begin{proof}
\eqref{t:elem reg epi f i}:~First observe that since 
\begin{align*}
\ncone{\epi f}(\xbar) &\supseteq \left\{(v,-1)~|~ v\in\lsd f(\xbar)  \right\}
\cup\left\{(v,0)~|~ v\in\lsd_{\infty} f(\xbar)  \right\},\text{~and}\\
\cncone{\epi f}(\xbar) &\supseteq \left\{(v,-1)~|~ v\in\csd f(\xbar)  \right\}
\cup\left\{(v,0)~|~ v\in\csd_{\!\!\!\!\infty} f(\xbar)  \right\},
\end{align*}
any point $(\xbar, \vbar)\in(\gph\lsd f\cup \gph\lsd_{\infty} f)$ corresponds to either a normal vector 
of the form $(\vbar, -1)$ or a horizon normal of the form $(\vbar, 0)$.  Suppose first that $f$ is $\varepsilon$-epi-subregular at $\xbar$ relative to $\Lambda\subset\dom f$ for $\vbar\in\csd f(\xbar)$
with constant $\veps$ and neighborhood $U'$ of $\xbar$.  Then $\epi f$ is $\veps$-subregular 
 at $(\xbar, f(\xbar))$ relative to  $\epi (f_\Lambda)$ for $(\vbar, -1)\in \cncone{\epi f}(\xbar, f(\xbar))$ with 
 constant $\veps$ and neighborhood $U$ of $(\xbar,f(\xbar))$ in \eqref{e:esr}. Thus,  for all $(x,\alpha)\in \epi (f_\Lambda)\cap U$, we have
 \begin{align*}
\left\langle(\vbar, -1), (x,\alpha)-(\xbar,f(\xbar))  \right\rangle &\leq \veps\|(\vbar, -1)\|\|(x,\alpha)-(\xbar,f(\xbar))\|
\\
\iff \left\langle\vbar, x-\xbar  \right\rangle - \alpha+f(\xbar) &\leq \veps\left(\|\vbar\|^2+1\right)^{1/2}\left(\|x-\xbar\|^2+(\alpha-f(\xbar))^2\right)^{\frac{1}{2}}\\
&=\veps\|\vbar\|\|x-\xbar\|\left(1+\|\vbar\|^{-2}\right)^{\frac{1}{2}}\left(1+(\alpha-f(\xbar))^2\|x-\xbar\|^{-2}\right)^{\frac{1}{2}},
 \end{align*}
which from the claim follows.

The only other case to consider is that $f$ is $\varepsilon$-epi-subregular at $\xbar$ relative to 
$\Lambda\subset\dom f$ for $\vbar\in\csd_{\infty} f(\xbar)$
with constant $\veps$ and neighborhood $U'$ of $\xbar$.  In this case,   $\epi f$ is $\veps$-subregular 
 at $(\xbar, f(\xbar))$ relative to  $\epi (f_\Lambda)$ for $(\vbar, 0)\in \cncone{\epi f}(\xbar, f(\xbar))$ with  constant $\veps$ and neighborhood $U$ of $(\xbar,f(\xbar))$ in \eqref{e:esr}. Thus, for all $(x,\alpha)\in \epi (f_\Lambda)\cap U$, we have
 \begin{align*}
 	&&\left\langle(\vbar, 0), (x,\alpha)-(\xbar,f(\xbar))  \right\rangle &\leq \veps\|(\vbar, 0)\|\|(x,\alpha)-(\xbar,f(\xbar))\|	\\
 	&\iff& \left\langle\vbar, x-\xbar  \right\rangle &\leq \veps\|\vbar\|\left(\|x-\xbar\|^2+(\alpha-f(\xbar))^2\right)^{1/2}\\
 	&\iff &\left\langle\vbar, x-\xbar  \right\rangle&\leq  \veps\|\vbar\|\|x-\xbar\|\left(1+(\alpha-f(\xbar))^2\|x-\xbar\|^{-2}\right)^{1/2},
 \end{align*}
which completes the proof of \eqref{t:elem reg epi f i}. 

\eqref{t:elem reg epi f iii}:~Follows immediately from the definition. 
\end{proof}

\begin{remark}[indicator functions of subregular sets]\label{remark:iota subregular set}
When $f=\iota_\Omega$ for a closed set $\Omega$ the various subdifferentials coincide with the 
respective normal cones to $\Omega$.  In this case, inequality 
\eqref{eq:subregular horizon fns} subsumes \eqref{eq:subregular fns} since 
all subgradients are also horizon subgradients and \eqref{eq:subregular horizon fns} reduces to 
\eqref{e:esr} in agreement with the corresponding notions of regularity of sets.
\end{remark}

\bigskip

\noindent\textbf{Acknowledgments.}  The research of AD has been supported by the grants
AFB170001 (CMM) \& FONDECYT 1171854 (Chile) and MTM2014-59179-C2-1-P (MINECO of Spain). 
The research of DRL was supported in part by DFG Grant SFB755 and DFG Grant GRK2088. 
The research of MKT was supported in part by a post-doctoral fellowship from the Alexander von Humboldt Foundation.


\end{document}